\documentclass[12pt]{article}
\usepackage{mathrsfs}
\usepackage{bbm}

\usepackage{verbatim}

\usepackage{amsthm}
\usepackage{amsmath}
\usepackage{amssymb}

\usepackage{graphicx}
\usepackage{epstopdf}
\usepackage{titling}

\usepackage{enumerate}

\usepackage{algpseudocode}
\usepackage{algorithm}
\usepackage{algorithmicx}
\usepackage{tikz}

\newtheorem{theorem}{Theorem}[section]
\newtheorem{thm}[theorem]{Theorem}
\newtheorem{prop}[theorem]{Proposition}
\newtheorem{lem}[theorem]{Lemma}

\newtheorem{cor}[theorem]{Corollary}

\newtheorem{example}[theorem]{Example}
\newtheorem{conj}[theorem]{Conjecture}
\newtheorem{defn}[theorem]{Definition}

\begin{document}

\title{On high-dimensional acyclic tournaments}

\author{Nati Linial\thanks{School of Computer Science and engineering, The Hebrew University of Jerusalem, Jerusalem 91904, Israel. Email: {\tt nati@cs.huji.ac.il}. Research supported in 
part by the Israel Science Foundation and by a USA-Israel BSF grant.}
\and Avraham Morgenstern\thanks{Einstein Institute of mathematics, The Hebrew University of Jerusalem, Jerusalem 91904, Israel. Email: {\tt avraham.morgenstern@mail.huji.ac.il}}
}


\maketitle

\setcounter{page}{1}

\vspace{-2em}

\begin{abstract}
We study a high-dimensional analog for the notion of an acyclic (aka transitive) tournament. We give upper and lower bounds on the number of $d$-dimensional $n$-vertex acyclic tournaments. In addition, we prove that every $n$-vertex $d$-dimensional tournament contains an acyclic subtournament of $\Omega(\log^{1/d}n)$ vertices and the bound is tight. This statement for tournaments (i.e., the case $d=1$) is a well-known fact. We indicate a connection between acyclic high-dimensional tournaments and Ramsey numbers of hypergraphs. We investigate as well the inter-relations among various other notions of acyclicity in high-dimensional tournaments. These include combinatorial, geometric and topological concepts.
\end{abstract}

\section{Introduction}

A tournament is an orientation of a complete graph. The study of
tournaments is a classical topic in combinatorics. Already in the
1960's a whole monograph~\cite{moon} was dedicated to this subject.
Many theorems have been proved about tournaments over the years. Here
we take a geometric perspective of the subject and view a tournament
as an orientation of the one-dimensional skeleton of a simplex.
As it turns out, higher-dimensional analogs where we orient the
higher skeletons of the simplex are rich in structure and raise many
intriguing problems. To make the distinction clear, we often refer henceforth
to traditional tournaments as $1$-tournaments and to their $d$-dimensional
counterparts as $d$-tournaments.

As far as we know, the first paper on higher dimensional
tournaments is due to Leader and Tan~\cite{leader}. It is well-known
and easy to prove that in a $1$-tournament at most one quarter of the
triples are cyclic, and they investigate higher-dimensional
analogs of this statement.

We start with some definitions and background material.
Unless otherwise stated, every tournament that we consider has vertex set $V = [n]=\{1,\ldots,n\}$ with the natural order. Maintaining the topological terminology,
we refer to a subset $A \subseteq V$ as a {\em face} of {\em dimension} $|A|-1$. A face of dimension $d$ is called a $d$-face for short, or even just a face when the relevant dimension is clear from the context. A $d$-tournament $T = (V, \epsilon)$ on vertex set $V$ is specified by a mapping $\epsilon: {V\choose d+1} \to \{-1,1\}$. For a $d$-face $\sigma\in {V\choose d+1}$ we call $\epsilon(\sigma)$ the {\em orientation} of $\sigma$. We mostly write $\epsilon_{\sigma}$ rather than $\epsilon(\sigma)$.

For faces $\tau\subset\sigma$ of dimension $d-1,d$, respectively, define $(\tau;\sigma)$ as the orientation {\em induced} on $\tau$ by the positive orientation of $\sigma$ (viewed as a $d$-dimensional simplex). Namely, let $\sigma={i_0<i_1<\ldots<i_d}$, $\tau=\sigma\setminus\{i_j\}$, then $(\tau;\sigma)=(-1)^{d-j}$. Now if $\sigma$ is oriented, with orientation $\epsilon_\sigma$, the orientation induced on $\tau$ is $\epsilon_\sigma\cdot(\tau;\sigma)$. If $\tau\not\subset\sigma$ we define $(\tau;\sigma)$ to be zero.

The {\em incidence matrix} of a $d$-tournament $T$ is an ${n \choose d} \times {n \choose d+1}$ matrix $A$ whose rows
and columns correspond to all subsets of $V$ of cardinality $d$ resp. $d+1$.
The $(\tau, \sigma)$ entry of $A$ is the orientation induced from $\sigma$ to $\tau$ (and zero if $\tau \not \subset \sigma$). Clearly, $T = (V, \epsilon)$ can be read off the incidence matrix $A$. We often do not distinguish between a face and the corresponding (column) vector of the incidence matrix. Note that for $d=1$ these definitions yield the traditional definitions of a tournament and its incidence matrix.

In order to deal with {\em partial tournaments} we allow $\epsilon$ to take the value $0$ as well. In that case, if $\epsilon_{\sigma}=0$, the $\sigma$-column of the incidence matrix is an all-$0$ column. 

\begin{figure}
 \centering
 \includegraphics[width=100mm,height=60mm]{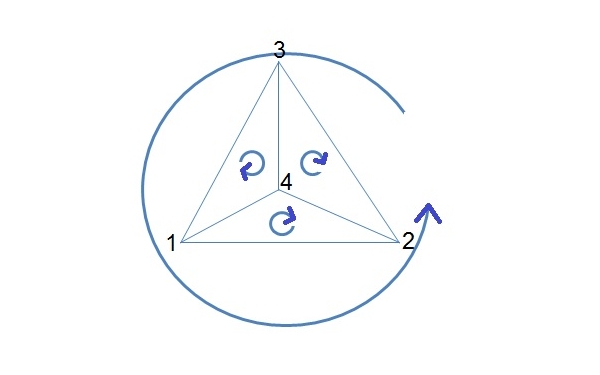} 
 \caption{A $2$-tournament on four vertices}
 \label{fig:1}
\end{figure}

\begin{example}\label{ex1}
Consider the $2$-dimensional tournament $T=([4],(\epsilon_{123},\epsilon_{124},\epsilon_{134},\epsilon_{234}))=$ \newline $([4],(1,-1,1,-1))$. Figure \ref{fig:1} demonstrates the orientation of each of the four faces. The corresponding incidence matrix is 
$$\bordermatrix{
~ & 123 & 124 & 134 & 234 \cr
12 & 1 & -1 & 0 & 0 \cr
13 & -1 & 0 & 1 & 0 \cr
14 & 0 & 1 & -1 & 0 \cr
23 & 1 & 0 & 0 & -1 \cr
24 & 0 & -1 & 0 & 1 \cr
34 & 0 & 0 & 1 & -1 \cr
}$$
\end{example}

Let $T=(V,\epsilon)$ be a $d$-tournament with incidence
matrix $A$ and let $x \in V$. The {\em link} of $x$ in $T$
denoted $lk_T(x)$ is a $(d-1)$-tournament on vertex set $V \setminus \{x\}$. It assigns to a $(d-1)$-face $\tau$ the orientation that is induced on $\tau$ by $\tau\cup\{x\}$ in $T$, namely, $\epsilon_{\tau\cup\{x\}}\cdot(\tau;\tau\cup\{x\})$.

The {\em degree sequence} of the tournament is the vector $A\cdot \vec{1}$ where $\vec{1}$ is the vector of $1$'s of length ${n\choose d+1}$. (Note that this definition deviates a little from the standard $1$-dimensional definition. When $d=1$, the $x$-entry of this vector is $s^{+}(x)-s^{-}(x)$, where $s^{\pm}(x)$ is the number of outgoing/incoming edges for vertex $x$). The degree sequence is a sequence of integers in the range $[-(n-d),n-d]$, all of which have the same parity as $n-d$.

A non-empty collection $C$ of faces in a $d$-tournament $T$ is called a {\em cycle} if there is a real vector $v$ with nonnegative entries such that $Av=0$, whose support (i.e., the index set of the positive coordinates in $v$) coincides with $C$. In other words there are {\em positive} real number $v_F$ for every $F\in C$ s.t. 
\begin{equation}\label{def:cycle}
\sum_{F\in C}v_F\cdot F=0.
\end{equation} 
where we identify a face with the corresponding column of the incidence matrix. E.g. The set $\{123, 124, 134, 234\}$ is a cycle in the tournament $T$ of Example \ref{ex1} (all coefficients equal $1$). 

{\em Transitive} (={\em acyclic}) $1$-tournaments and subtournaments are thoroughly studied, and
here is a $d$-dimensional counterpart of this notion:
\begin{defn} A tournament $T$ is {\em acyclic} or {\em cycle-free} if it contains no cycles. \end{defn}

Clearly, there are exactly $n!$ acyclic $1$-tournaments on $n$ vertices.
In Section~\ref{sec:enum} we study the number of $n$-vertex acyclic $d$-tournaments and show (Theorem~\ref{thm1}) that it is $n^{\Theta(n^d)}$. The proof(s) involve both analytic and geometric ideas. In particular, it is easy to tell from the degree vector of a $1$-tournament whether or not the tournament is acyclic. As we show (Lemma~\ref{lem:deg_seq}) it is possible to decide whether a $d$-tournament is acyclic by observing its degree sequence.

Let $T$ a partial $d$-tournament. Using a term from the topology of simplicial complexes, a $(d-1)$-face $F$ is called {\em free} if all the $d$-faces that contain it induce the same orientation on $F$. (I.e., the corresponding row in $T$'s incidence matrix is either non-negative or non-positive). An {\em elementary collapse} is a step in which we pick a free $(d-1)$-face $F$ and remove from $T$ all the $d$-faces containing it, i.e., we set $\epsilon_G=0$ for all $G\supset F$. We call $T$ {\em collapsible} if it is possible to arrive at $\epsilon=\vec{0}$ in a series of elementary collapses. Note, e.g., that the $2$-tournament of Example \ref{ex1} has no free faces and is, therefore, not collapsible. In contrast,

\begin{example}

The incidence matrix of the $2$-tournament $T=([4],(1,1,1,1))$ is 
$$\bordermatrix{
~ & 123 & 124 & 134 & 234 \cr
12 & 1 & 1 & 0 & 0 \cr
13 & -1 & 0 & 1 & 0 \cr
14 & 0 & -1 & -1 & 0 \cr
23 & 1 & 0 & 0 & 1 \cr
24 & 0 & 1 & 0 & -1 \cr
34 & 0 & 0 & 1 & 1 \cr
}$$
It is easily verified that $T$ is collapsible. Note, e.g., that the face $13$ is not free, but becomes free once we collapse, e.g., the face $12$.

\end{example}

It is easy to see that a collapsible tournament must be acyclic. No subface of a face that participates in a cycle can be free. This remains so even following any sequence of elementary collapses.

Let us recall the following well-known fact about $1$-tournaments \cite{er:moser}:

\begin{thm} Every $1$-tournament on $n$ vertices has an acyclic subtournament on $\log_2 n$ vertices. There exist $1$-tournaments with no acyclic subtournament on $(2+o(1))\log_2 n$ vertices. This, in particular,
holds for random $1$-tournaments.\end{thm}

In Section~\ref{sec:sub} we derive a $d$-dimensional analog of this theorem. We show that every $n$-vertex $d$-tournament has an acyclic subtournament on $\Omega({\log^{1/d} n})$
vertices and the bound is tight.

There are several simple conditions on $1$-tournaments which are all equivalent to ``transitivity''. Namely, (i) $T$ contains no cyclic triangles, (ii) $T$ contains no (graph-theoretical) directed cycle, (iii) $T$ is acyclic as defined above, (iv) $T$ is collapsible, and finally, (v) all edges of $T$ go forward relative to some total order on the vertices. In other words, the vertices can be mapped to $\mathbb R$ with all edges going from left to right.

In Section~\ref{sec:equiv} we study the implications among these notions in high-dimensional tournaments and we observe that for $d>1$, the implications (v)$\implies$(iv)$\implies$(iii)$\implies$(ii)$\implies$(i) hold. We construct examples which show that all the reverse implications do not hold. This paper raises many open questions, and in Section~\ref{sec:final} we describe a few additional directions for further research. 

\subsection{Hyperplane Arrangements}

A hyperplane arrangement $\mathcal A$ in $\mathbb{R}^n$, or simply an {\em arrangement} is a set of hyperplanes in $\mathbb{R}^n$. A {\em chamber} of $\mathcal A$ is a 
connected component of $\mathbb{R}^n\setminus(\cup_{H \in \mathcal A} H)$. The {\em braid arrangement} is a famous example which is of relevance to us. Its hyperplanes are $H_{ij} = \{x \in\mathbb{R}^n ~|~ x_i-x_j=0\}$ for $1\le i<j\le n$. Its relevance to our discussion comes from the simple bijection between the chambers of the braid arrangement and permutations in $S_n$, or, what is the same,
acyclic $n$-vertex $1$-tournaments. For
a comprehensive survey of arrangements, see \cite{stanley}.

As we observe below, there is a natural bijection between $n$-vertex acyclic $d$-tournaments and the chambers of a certain arrangement in $\mathbb{R}^{n\choose{d}}$. This arrangement may be of independent interest for other reasons as well, as we explain in Section~\ref{sec:final}.

\section{Enumerating acyclic tournaments}\label{sec:enum}

We denote by $a_d(n)$ the number of acyclic $n$-vertex $d$-tournaments.

\begin{thm}\label{thm1}
For every integer $d\ge 1$ and every large enough $n$ there holds
\[
a_d(n) \le \left(\frac{e}{d+1}n\right)^{{n \choose d}}.
\]
Also
\[
a_d(n) \ge\left(\frac n{e^{H_d}+o_n(1)}\right)^{n\choose d}
\]
where  $H_d$ is the harmonic sum $H_d=\sum\limits_{k=1}^d \frac{1}{k}$. In particular, for large $d$, 
\[a_d(n)\ge \left(\frac{e^{-\gamma}+o_d(1)}{d}\cdot n\right)^{n\choose d}\] 
where $\gamma=0.577\ldots$ is the Euler constant.
\end{thm}

\begin{proof}\let\qed\relax

We start with the lower bound which is a consequence of the following inequality on $a_d(n)$. This inequality ties between the cycles of a tournament, and the cycles of its (vertex) links.

\begin{equation}\label{recursion}
\forall~n\ge d\ge 2 \;\;\;\;\;a_d(n)\geq\prod\limits_{k=d-1}^{n-1}a_{d-1}(k).
\end{equation}

Before we prove Inequality~(\ref{recursion}), we use it to derive the lower bound in Theorem~\ref{thm1}. For $d=1$ the lower bound follows from Stirling's formula, since $a_1(n)=n!=\prod_1^n k$.
We proceed to larger $d$'s. The inequality implies that $a_2(n)\ge \prod\limits_{k=1}^{n-1} k^{n-k}$, and more generally, by induction, that
\[a_{d}(n)\ge \prod\limits_{k=1}^{n-d+1} k^{n-k \choose d-1},\] 
In particular,
\[a_d(n)\ge n^{-O(n^{d-1})}\prod\limits_{k=1}^{n} k^{n-k \choose d-1}.\]
By sweeping more of the error terms into the expression $n^{O(n^{d-1})}$, we can further write
\[
a_d(n)\ge n^{-O(n^{d-1})}\prod\limits_{k=1}^{n} k^{n-k+d-2 \choose d-1}\ge 
n^{-O(n^{d-1})}\prod\limits_{k=1}^{n} k^{(n-k)^{d-1}/(d-1)!}.
\]
(In the first inequality we gave up a factor of $\prod_1^n k^{{n-k+d-2 \choose d-1}-{n-k \choose d-1}} \le \prod_1^n k^{O(n^{d-2})} \le n^{O(n^{d-1})}$.)

Consequently,
\begin{align*}
\log(a_d(n)) &\ge \sum_1^n \frac{(n-k)^{d-1}}{(d-1)!}\log k -O(n^{d-1} \log n) \\
&\ge{\frac{1}{(d-1)!}\int\limits_{1}^{n} (n-x)^{d-1}\log x~dx} -O(n^{d-1} \log n).\\
\end{align*}

The integral estimate for the sum follows from the fact that the error term is as large as the maximum of the function over the range of integration.

Using the binomial formula,
\[
\int (n-x)^{d-1}\log x~dx = \sum\limits_{r=0}^{d-1} {d-1 \choose r} (-1)^r n^{d-1-r} x^{r+1} (\frac{\log x}{r+1}-\frac{1}{(r+1)^2}).
\] 
This gives
\[\frac{\log(a_d(n))}{{n\choose d}}\ge d\sum\limits_{r=0}^{d-1} {d-1 \choose r} (-1)^r (\frac{\log n}{r+1}-\frac{1}{(r+1)^2}) - O(\frac{\log n}{n}).
\]

It only remains to verify the simple identities

\[\sum\limits_{r=0}^{d-1} {d-1 \choose r} (-1)^r \frac{1}{r+1}=\frac{1}{d}\]

and 

\[\sum\limits_{r=0}^{d-1} {d-1 \choose r} (-1)^r \frac{1}{(r+1)^2} = \frac{1}{d}(1+\frac{1}{2}+\frac{1}{3}+\ldots+\frac{1}{d}).\]\end{proof}

We now turn to prove Inequality~(\ref{recursion}).

\begin{prop}\label{prop1} Let $T$ be a $d$-tournament. Suppose that for every vertex $n \ge i\ge 1$ the $(d-1)$-tournament $lk_T(i)$ is acyclic, then $T$ is acyclic.
Moreover, the same conclusion holds even if we only assume that the
restriction of the link $lk_T(i)$ to $\{i+1,i+2,\ldots,n\}$ is acyclic. \end{prop}

\begin{proof}
We only state the proof of the second, stronger part of the proposition.
Assume to the contrary that $T$ contains a cycle $C$ and
\begin{equation}\label{eqn:link}
\sum_{F \in C} v_F F =0
\end{equation} with $v_F > 0$ for all $F \in C$ (see Equation~(\ref{def:cycle})). Now let $n\ge k\ge 1$ be the lowest index of a vertex in $\cup_{F\in C}F$. Let $D:=\{F|~k\in F\in C\}$. For $F \in D$ we let $F':=F\setminus\{k\}$, and we claim that $\sum_{F \in D} v_F F' =0$, contrary to our assumption that the
restriction of $lk_T(k)$ to $\{k+1,k+2,\ldots,n\}$ is acyclic. To see this, note that $B$, the incidence matrix of $lk_T(k)|_{\{k+1,\cdots,n\}}$, is, possibly with a global sign reversal, a submatrix of $T$'s incidence matrix $A$. 

We write $A$ in block form as follows

$$\bordermatrix{
~ & k\not\in F & k\in F; \min(F)< k & \min(F)=k \cr
~~~~~~~~~~~~~~~~k\not\in\tau & X_1 & X_2 & X_3\cr
k\in\tau; \min(\tau)< k & 0 & X_4 & X_5 \cr
~~~~~~~~\min(\tau)=k & 0 & X_6 & X_7  \cr
}$$
where the rows are indexed by $(d-1)$-faces $\tau$, and the columns are indexed by $d$-faces $F$. Note that $v$ may be viewed as a vector in the right kernel of $A$. It can be expressed in corresponding block form as $v=(v_1, 0, v_2)$, where $v_2\neq 0$ by definition of $k$. It follows that $v_2$ is in the right kernel of $X_7$ which proves our claim, since $X_7=\pm B$. \end{proof}

We are now ready to complete the proof of Inequality~(\ref{recursion}), by
providing a scheme that yields many acyclic $d$-tournaments $T$ on
vertex set $[n]$. Select first an arbitrary acyclic $(d-1)$-tournament on vertex set $[2,n]$ to be $lk_T(1)$. Then an acyclic $(d-1)$-tournament on vertex set $[3,n]$ to be $lk_T(2)|_{[3,n]}$ etc. By Proposition~\ref{prop1} the resulting $d$-tournament $T$ is indeed acyclic. The desired inequality follows.

This concludes the proof of the lower bound and we now turn to prove the upper bound. We first note that acyclic $d$-tournaments are uniquely determined by a their degree sequence. This simple observation gives an upper bound that is weaker than what is stated in the theorem. We still find it worthwhile to state, since several interesting questions arise in this context.

\begin{thm}\label{easy_ub} \[a_d(n)\le n^{n \choose d}.\]
\end{thm}

\begin{proof}
The degree sequence of a $d$-tournament is an $n\choose d$-vector whose entries have the parity of $n-d$ and reside in  $[-(n-d),n-d]$. There are $(n-d+1)^{n\choose d}$ such vectors. The following lemma completes the proof.
\end{proof}

\begin{lem}\label{lem:deg_seq}
An acyclic $d$-tournament is uniquely determined by its degree sequence.
\end{lem}

\begin{proof}
Let $T$ and $S$ be two acyclic $d$-tournaments with respective incidence matrices $A,B$ s.t. $A\cdot\vec{1}=B\cdot\vec{1}$. Define a vector $v\in\mathbb{R}^{n\choose d+1}$ as follows
\[
v_c=\begin{cases}
0,\;\;\;\; A_{*,c}=B_{*,c}\\
1,\;\;\;\; A_{*,c}=-B_{*,c}
\end{cases}
\]
 Then $A\cdot v= \frac{1}{2}(A-B)\cdot v =\frac{1}{2} (A-B)\cdot \vec{1}=0$, so that either $A=B$ or $v \neq 0$ and we found a cycle in $T$. 
\end{proof}

This discussion raises several interesting questions concerning degree sequences. In particular we can ask
\begin{itemize}
\item
How many distinct degree sequences there are to $d$-dimensional $n$-vertex tournaments? For $d=1$ quite a lot is known~\cite{kleitman:winston}. It would also be interesting to get some characterizations, efficient ways to recognize such sequences etc.
\item
Of course, all acyclic $n$-vertex $1$-tournaments have the same degree sequence, up to permutation. It seems quite intriguing to understand the degree sequences of acyclic $d$-dimensional tournaments for $d > 1$.

\end{itemize}

\subsection{An improved upper bound using arrangements}\label{sub:arr}

In the same way that acyclic $1$-tournaments are related to the braid arrangement, there are higher-dimensional counterparts to this arrangement that correspond to $d$-dimensional acyclic tournaments. The arrangement in question is ${n\choose d}$-dimensional and has one hyperplane for each $d$-face.
The hyperplane $H_\sigma$ corresponding to the $d$-face $\sigma=\{i_0<\ldots<i_d\}$ is defined by the equation
$$\sum_{\tau} (\tau;\sigma)x_{\tau}=\sum\limits_{k=0}^{d} (-1)^{d-k} x_{\sigma \setminus \{i_k\}}=0.$$
(where the coordinates of the vectors in $\mathbb{R}^{n\choose d}$ are indexed by ${[n] \choose d}$).

There is a natural bijection between chambers of this arrangement and acyclic $d$-tournaments on vertex set $[n]$: Corresponding to a chamber $C$ of the arrangement is the tournament that orients the face $\sigma={i_0<\ldots<i_d}$ according to the rule \[\epsilon_\sigma=\text{sgn}(\sum\limits_{k=0}^{d} (-1)^k x_{\sigma\setminus\{i_k\}}),\] where $x$ is an arbitrary point in $C$. It is easy to see that the orientation does not depend on the choice of $x \in C$.

In the opposite direction, we want to associate a chamber $C$ to a given acyclic tournament $T=([n],\epsilon)$. Equivalently, it suffices to specify a point $x \in C$. This means that we must show the consistency of the following system of inequalities
\[f_{\sigma}(x):=\epsilon_\sigma\sum\limits_{k=0}^{d} (-1)^k x_{\sigma\setminus\{i_k\}}>0\mbox{ for every $d$-face~~}\sigma=\{i_0<\ldots<i_d\}.\] 
By linear programming duality this system is inconsistent iff there exists a nonnegative linear combination of the $f_{\sigma}$ that is identically zero. Namely, there exist $\alpha_\sigma\geq0$ not all zero, s.t. $\sum_{\sigma}\alpha_{\sigma} f_{\sigma} = 0$. But such $\alpha_{\sigma}$ constitute the coefficients of a cycle in $T$.

We next recall the well-known fact (e.g.,~\cite{matousek}) that an $n$-dimensional hyperplane arrangement with $m$ hyperplanes has at most $\sum\limits_{k=0}^n {m\choose k}$ chambers. We can now complete the proof of the upper bound in Theorem \ref{thm1}.
For $n$ large enough,
\[
a_d(n) \le \sum\limits_{k=0}^{n\choose d} {{n\choose d+1}\choose k}\le 2^{{n\choose d+1}H({n\choose d}/{n\choose d+1})} = 2^{\frac{n-d}{d+1}H(\frac{d+1}{n-d}){n\choose d}}\]
It is not hard to verify that for $1 \ge x \ge 0$ the binary entropy function satisfies $H(x)\le x\cdot\log_2(\frac{e}{x})$, which yields
\[a_d(n)\le \left(\frac{n-d}{d+1}\cdot e\right)^{n\choose d}.\]

The claim follows.\qed

\section{Large acyclic subtournaments}\label{sec:sub}

\begin{thm}\label{sub_acyc} Every $d$-tournament on $n$ vertices has an acyclic subtournament on $\Omega(\log^\frac{1}{d}(n))$ vertices. This bound is tight up to a constant factor, and is attained, in particular, by random $d$-tournaments.
\end{thm}

\begin{proof} We will go through the $(d-1)$-faces in their reverse lexicographic order and eliminate some vertices along the way. Consider the current $(d-1)$-face $\tau$ and the set $S$ of all currently remaining vertices that precede all the vertices of $\tau$. We delete some of the elements $x \in S$ according to the following criterion. The $d$-face $\sigma=\tau\cup\{x\}$ has its orientation $\epsilon_{\sigma}$ and it induces an orientation on $\tau$. This splits $S$ into two parts according to the orientation induced on $\tau$. We eliminate all the vertices in the smaller of these two parts and all $(d-1)$-faces that contain an eliminated vertex.

We make two claims:
\begin{itemize}
\item
The remaining tournament is collapsible, and hence acyclic.
\item
At least $\Omega(\log^{1/d} n)$ vertices survive the whole process.
\end{itemize}

To prove the first claim, note that the minimal face (in reverse lexicographic order) is free. After that face is being collapsed, the next minimal face becomes free once again, etc.

For the second claim note that at each step, the size of the remaining vertex set is at least a half of its previous size. Let $K$ be the set of vertices that survive the whole process, and let $|K|=k$. The collection of $(d-1)$-faces $\tau$ that are examined in the process is exactly ${K \choose d}$. Consequently, $n/2^{k\choose{d}}\leq k$, which yields the claimed bound $k \ge \Omega(\log^{1/d} n)$.

Tightness follows from Theorem~\ref{thm1} combined with a simple first moment argument. Fix integers $n, k$ and $d$ and consider a random $d$-tournament $T$ on $n$ vertices. Let $X$ be the random variable that counts the number of acyclic $k$-vertex subtournaments of $T$. It follows from Theorem~\ref{thm1} that
\[
\mathbb{E}(X) \le {n\choose k}\frac{k^{O(k^d)}}{2^{k\choose d+1}}.
\]
Consequently, there is a value of $k\le O(\log^{1/d} n)$ for which $\mathbb{E}(X)<1$. The conclusion follows.
\end{proof}

It is well known and easy to show~\cite{er:moser} that every $n$-vertex $1$-tournament contains an acyclic subtournament on $\log_2 n$ vertices and that in a random $1$-tournament the largest acyclic subtournament has $(2+o(1))\log_2 n$ vertices. However, despite many attempts, it seems difficult to close this gap. We therefore suspect that closing the gap between the upper and the lower bound in Theorem~\ref{sub_acyc} will not be an easy task.

\subsection{A connection with Ramsey theory}\label{sec_ramsey}

As usual, we denote
by $R_d(l,k)$ the smallest integer $n$ for which the following holds. Every red/blue coloring of the hyperedges in the complete $n$-vertex $d$-uniform hypergraph contains either a complete $l$-vertex red hypergraph or a complete $k$-vertex blue hypergraph. Relatively little is known about the growth rate of these numbers for $d > 2$. In a recent paper~\cite{conlon}, Conlon, Fox and Sudakov ask in particular, whether $R_d(d+1,k)$ grows like a tower of height $(d-1)$ in $k$ (i.e., $2^{2^{2^{2^{\cdots^{2^k}}}}}$). We are unable to answer their question, but we note that the notion of acyclic $d$-tournaments allows us to extend an old argument of Erd\H{o}s and Hajnal~\cite{er:haj1} and show

\begin{thm}\label{ramsey}
For all $d\ge 1$, $R_{d+2}(d+3,k)\ge 2^{c_d k^d}$. 
\end{thm}

The smallest cyclic $d$-tournament has $d+2$ vertices. We refer to this tournament simply as a {\em $(d+2)$-cycle} and we note that on a given set of $d+2$ vertices there are exactly two possible $(d+2)$-cycles. In particular, the probability that a random $(d+2)$-vertex $d$-tournament is a $(d+2)$-cycle is $2^{-d-1}$. 

\begin{lem}\label{no_d+2_cycle}
The probability that a random $n$-vertex $d$-tournament contains no $(d+2)$-cycle is at most $2^{-\Omega(n^{d+1})}$.
\end{lem}

\begin{proof}
The Erd\H{o}s-Hanani Conjecture was proved by R\"{o}dl (e.g.,~\cite{prob_method}). It implies the existence of a large system of $(d+2)$-sets of vertices no two of which have $d+1$ vertices in common. Specifically, there exists such a system of $(1-o_n(1))\frac{{n \choose d+1}}{d+2}=(1-o_n(1))\frac{n^{d+1}}{(d+2)!}$ sets. As noted above, each member of this system is a cycle with probability $2^{-d-1}$, and the claim follows, since these events are independent.  
\end{proof}

\begin{cor}\label{no_cycle}
There exist $n$-vertex $d$-tournaments in which every subtournament on $c'_d\log^\frac{1}{d} n$ vertices contains a $(d+2)$-cycle. This in particular holds with positive probability for random $d$-tournaments. Here $c_d' > 0$ is a constant that depends only on $d$.
\end{cor}

\begin{proof}
The claim follows from a first-moment argument.
Let $T$ be a random $n$-vertex $d$-tournament, and let $X$ be the number of $k$-vertex subtournaments of $T$ that contain no $(d+2)$-cycle. By the previous lemma 
\[\mathbb{E}(X)\le{n \choose k}2^{-\Omega(k^{d+1})}\] For $k=c'_d\log^\frac{1}{d}n$ this expectation is less than $1$ and the claim follows.
\end{proof}

We can complete now the proof of Theorem~\ref{ramsey}. A $d$-tournament as in Corollary~\ref{no_cycle}, induces a red/blue coloring of the subsets $U\in {[n]\choose d+2}$ as follows. We color $U$ blue if it is a $(d+2)$-cycle, and red otherwise. The claim follows since no set of $d+3$ vertices is entirely blue. To see this, let $S$ be a set of $d$ vertices where $x_1, x_2, x_3$ are the three remaining vertices. Consider the orientation induced on $S$ by each of the three faces $S\cup\{x_i\}$ for $i=1,2,3$. There must be two of these orientations, say for $i=1,2$ for which the orientations on $S$ coincide. But then $S\cup \{x_1, x_2\}$ is not a $(d+2)$-cycle.
\qed

\section{Alternative notions of acyclicity}\label{sec:equiv}

We find it instructive to recall now the one-dimensional situation and see how things change as the dimension grows.
Indeed all of the following properties of a $1$-tournament $T$ are easily seen to be equivalent. 

\begin{enumerate}
\item\label{triangle}
$T$ contains no cyclic triangles.
\item\label{0/1}
$T$ contains no (graph-theoretical) directed cycle.
\item\label{acyclic}
$T$ is acyclic as defined above.
\item\label{collapse} 
$T$ is collapsible.
\item\label{realization}
All edges of $T$ go forward relative to some total order on the vertices.
In other words, the vertices can be mapped to $\mathbb R$ with all edges going from 
left to right.  
\end{enumerate}

As noted below, all the above properties of a tournament $T$ have $d$-dimensional counterparts as follows.

\begin{enumerate}
\item\label{d-triangle}
$T$ contains no $(d+2)$-cycle.
\item\label{d-0/1}
There is no nonempty set of faces in $T$ that sum to zero. In other words, zero is the only solution of Equation (\ref{def:cycle}) in $0/1$ coefficients. 
\item\label{d-acyclic}
$T$ is acyclic.
\item\label{d-collapse} 
$T$ is collapsible.
\item\label{d-realization}
Fix the positive orientation on $\mathbb{R}^d$.
The orientation of the faces of $T$ is induced from some general-position embedding of its vertices in $\mathbb{R}^d$.
\end{enumerate}

As we presently note, these conditions appear in increasing order of strength. We subsequently present examples that show that reverse implications need not hold.

\begin{prop}
The following implications among $d$-dimensional tournaments hold:
\ref{d-realization}$\Rightarrow$\ref{d-collapse}$\Rightarrow$\ref{d-acyclic}$\Rightarrow$\ref{d-0/1}$\Rightarrow$\ref{d-triangle}
\end{prop}
\begin{proof}

Most implications are very easy to verify and we prove here only the implication \ref{d-realization}$\Rightarrow$\ref{d-collapse} (that \ref{d-collapse}$\Rightarrow$\ref{d-acyclic} was already mentioned before). Let $T$ be a $d$-tournament that is realizable by means of an embedding $\iota$ of the vertex set in ${\mathbb R}^d$. Every $(d-1)$-face on the boundary of the convex hull $\text{conv}(\text{image}(\iota))$ is free. Once those are eliminated, the new boundary $(d-1)$-faces are free again, etc. 
\end{proof}

We turn to show that the reverse implications do not hold. We note that it suffices to consider $2$-dimensional examples to this end.

\begin{prop}
\ref{d-triangle}$\not\Rightarrow$\ref{d-0/1}
\end{prop}
\begin{proof}
Consider the three-dimensional octahedron $\text{conv}(\pm e_1, \pm e_2, \pm e_3)$. We orient its eight triangular facets according to the outer normal of this polytope. These eight faces sum to zero, hence this $2$-tournament is $0/1$-cyclic. Each of the remaining twelve $2$-faces contains an edge of the form $[-e_i,e_i]$ for some $3 \ge i \ge 1$. We orient these faces so that the orientation induced on this edge is $-e_i\to e_i$. Let us show that no set of four vertices can form a $d+2=4$-cycle. Every set of four vertices must contain at least one of the pairs $\{-e_i,e_i\}$ and cannot, therefore, be cyclic.
\end{proof}

\begin{prop}
\ref{d-0/1}$\not\Rightarrow$\ref{d-acyclic}
\end{prop}
\begin{proof}
Start with the standard $6$-point triangulation of the projective plane in Figure~\ref{fig:2}, where each of the $10$ faces is oriented clockwise. Now add the face $\sigma=\{1,2,3\}$ with the orientation $\epsilon_{\sigma}=+1$ i.e., $1\to 2\to 3\to 1$. Note that the sum of the $10$ faces $+2\sigma$ is zero. Hence this partial tournament is already cyclic, but contains no cycle in $0/1$ coefficients. We could try and orient the remaining $9$ faces so as to maintain the property that there is no cycle in $0/1$ coefficients, but this plan must fail. Consider the face $\rho=\{3, 4, 5\}$ and note that both $\pm \rho$ are expressible as $0/1$ combinations of already oriented faces, namely
\[ \rho=\sigma + \{1,3,4\} + \{1,4,5\} + \{1,2,5\} + \{2,3,5\}\]
\[-\rho=\sigma + \{2,3,4\} + \{2,4,6\} + \{4,5,6\} + \{3,5,6\} + \{1,2,6\} + \{1,3,6\}.\]
Consequently, however we orient $\rho$, a $0/1$-cycle is created.

As it turns out, this plan does work if we slightly modify the above construction. Start instead from the $10$ point triangulation in Figure~\ref{fig:2} along with the face $\sigma=\{1,2,3\}$. Again the $18$ faces in this triangulation are oriented clockwise and $\epsilon_{\sigma}=+1$. Consequently,
\begin{itemize}
\item
This partial tournament has a single cycle, {\em not} in $0/1$ coefficients.
\item
No additional cycle can be created by the addition of any single oriented face.
\end{itemize}
It follows that we can orient the remaining faces one by one so as to preserve the first property. Let $\rho$ be a face which wasn't oriented yet. Suppose that both orientations of $\rho$ are creating new cycles. Namely, $\sum v_F F +\rho=0$ and $\sum v'_F F -\rho=0$. Then $\sum (v_F + v'_F) F=0$ is another cycle, which doesn't involve $\rho$. Hence, this must be (a positive constant times) the only existing cycle. This means that the two new cycles created by $\pm\rho$ are using only the original faces and $\rho$, contrary to the above second property. 
\end{proof}

\begin{figure}
 \centering
 \includegraphics[width=130mm,height=60mm]{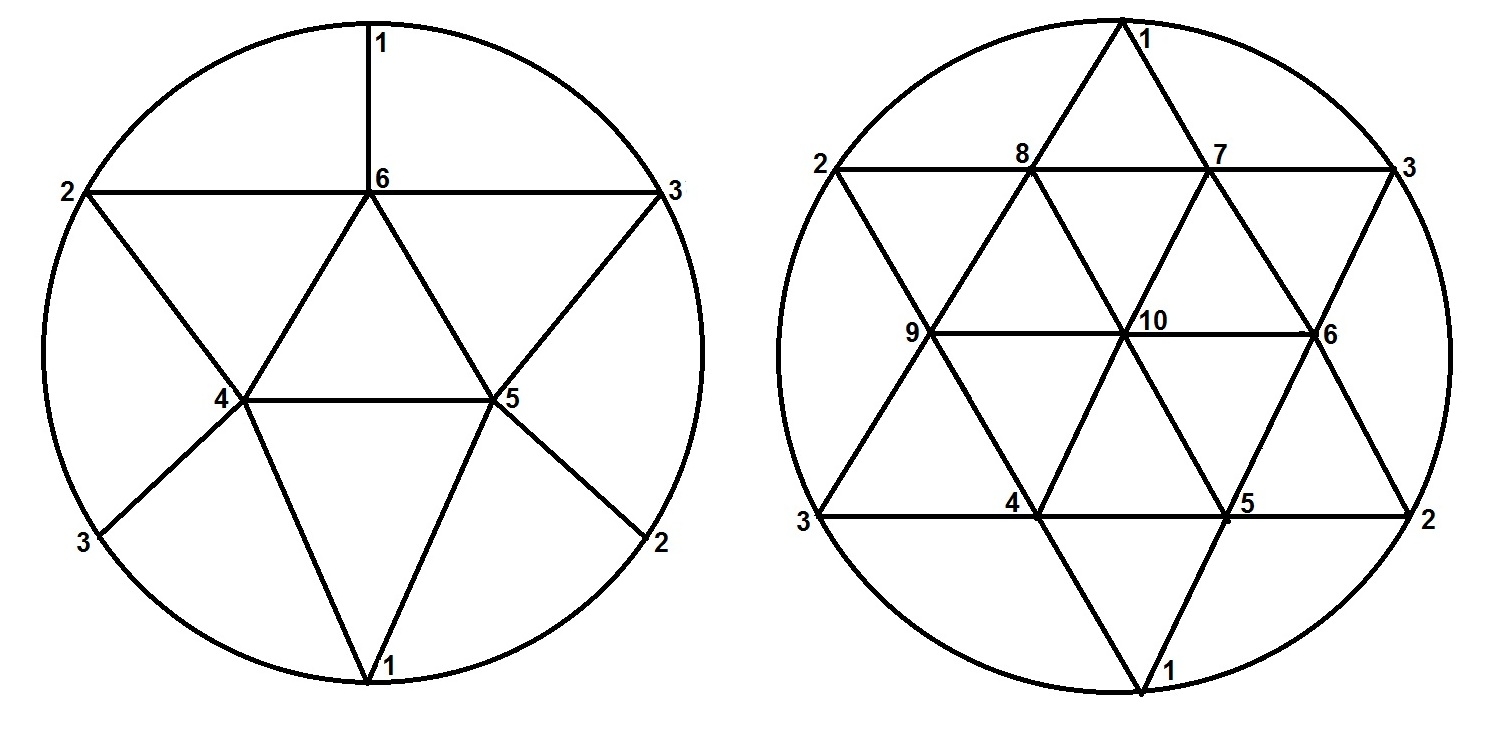} 
 \caption{The standard $6$ point triangulation of the real projective plane, and a modified $10$ point triangulation.}
 \label{fig:2}
\end{figure}

\begin{prop}
\ref{d-acyclic}$\not\Rightarrow$\ref{d-collapse}
\end{prop}
\begin{proof}
Using a computer search, we found a point $x\in{\mathbb R}^{{9\choose 2}}$ which doesn't satisfy any of the equations $x_{ij}-x_{ik}+x_{jk}=0$ ($i<j<k$). As explained above, such a point belongs to a unique chamber which corresponds to an acyclic $9$-vertex $2$-tournament. The resulting tournament is checked by the computer not to be collapsible. The point we found is \[x=(42, 0, 3, 88, 91, 87, 66, 28, 64, 60, 87, 11, 39, 81, 37, 51,  0, 23,\] \[77, 33, 23, 58, 11,  7, 70, 64, 73, 57, 86, 52, 98, 49, 57, 100, 43, 60).\] (Coordinates are indexed by unordered pairs and appear in lexicographic order).

Better still, consider the following point in ${\mathbb R}^{{10\choose 2}}$.
\[x=(76, 61, 70,  6, 95, 97, 45, 11, 26, 12, 33, 93,  5, 97, 92,  9, 48, 26, 58, 82,  4, 96,\] \[14, 83, 87, 92, 93, 92, 92, 18, 64, 11, 76,  4, 39, 82, 24, 94, 25, 36, 30, 40, 64, 21, 7).\]
This example is easier to verify, since the  acyclic $10$-vertex $2$-tournament corresponding to $x$ has no free faces.  
\end{proof}

\begin{prop}
\ref{d-collapse}$\not\Rightarrow$\ref{d-realization}
\end{prop}
\begin{proof}
We describe a collapsible $5$-vertex $2$-tournament which is not realizable in the plane.
We note first that there are precisely six isomorphism types of acyclic $5$-vertex $2$-tournaments, three of which are realizable in the plane.

All other three are collapsible. We prove this for a specific tournament: Let vertices $1,2,3$ be any triangle in the plane, with $4,5$ being mapped to the same point in the interior of the triangle (Think of $5$ as residing 'above' $4$). All faces, except for the three faces containing the edge $\{4,5\}$ are orientated clockwise. These three faces are oriented consistently with the orientation $4\to 5$ of the edge $\{4,5\}$. The resulting $2$-tournament is collapsible, since all the edges $\{4,5\}, \{1,2\},\{1,3\},\{2,3\}$ are free, and each face contains one of these edges. 

To see that this tournament is not realizable in the plane, note that its set of free edges form a disconnected graph, whereas in a realizable $d$-tournament the free $(d-1)$-faces form a (topological) cycle.
\end{proof}

Although we are mostly concerned here with acyclic $d$-tournaments, we find the other notions interesting as well. Below we make some comments about them, with a special interest in the relevant enumeration problems.

\subsection{Avoiding a $(d+2)$-cycle}

Leader and Tan's notion of a $d$-dimensional cycle \cite{leader} (which they call ``a directed simplex") coincides with our $(d+2)$-cycle.
In Section \ref{sec_ramsey} we make several observations concerning $(d+2)$-cycle-free $d$-tournaments. As mentioned, it is a classical fact that the largest acyclic subtournament of a random $1$-tournament has logarithmic order. Corollary~\ref{no_cycle} gives a $d$-dimensional analog of this statement, which is even stronger, since ``acyclic" is replaced by ``contains no $(d+2)$-cycle".

We give here upper bounds on the number of acyclic tournaments under the various interpretations of acyclicity, but we still do not know how many $n$-vertex $(d+2)$-cycle-free $d$-tournaments there are. The proof of Lemma~\ref{no_d+2_cycle} gives the upper bound
\[2^{\frac{1+o(1)}{d+2} {n\choose d+1}}\]
which far exceeds the other upper bounds proved here on the number of ``acyclic" tournaments. Whether or not this gap is inevitable we do not know.

\subsection{No $0/1$-cycles}
Theorem~\ref{easy_ub} is initially presented as an easy way to derive an upper bound on the number of acyclic tournaments. However its proof and specifically Lemma~\ref{lem:deg_seq} show the equivalence of the following two classes of $d$-tournaments: (i) Those that contain no cycle in $0/1$ coefficients, and (ii) Those that are uniquely reconstructible from their degree sequence.

In particular, Theorem~\ref{easy_ub} gives an upper bound on the number of $d$-tournaments with no $0/1$-cycles.

\subsection{Collapsible $d$-tournaments}
The lower bound in Theorem~\ref{thm1} actually applies to the smaller class of collapsible $d$-tournaments. To see this, note that inequality (\ref{recursion}) can be proved likewise if ``acyclic'' is replaced by ``collapsible''. It is interesting to get better upper and lower bounds on the number of collapsible $n$-vertex $d$-tournaments. In particular, better lower bounds will improve the lower bound in Theorem~\ref{thm1}. 

\subsection{Realizable $d$-tournaments} An ${\mathbb R}^d$-realizable $d$-tournaments is synonimous with an $n$-point {\em order type} in ${\mathbb R}^d$. Order types are of much interest in discrete geometry (see~\cite{matousek}). Their number is known to be $n^{(1+o(1))d^2 n}$~\cite{gp}.

\section{Some final comments and open questions}\label{sec:final}

We feel compelled
to recall the following well-known conjecture

\begin{conj}[Erd\H{o}s-Hajnal~\cite{er:haj2}]
For every graph $H$ there is a $\gamma > 0$ such that every $n$-vertex
graph which contains no induced copy of $H$ must have either a clique
or an anticlique of cardinality $\ge n^{\gamma}$.
\end{conj}
As shown in~\cite{solymosi} this conjecture can be restated as follows:
\begin{conj}
For every $1$-tournament $F$ there is a $\gamma > 0$ such that every
$n$-vertex $1$-tournament which contains no copy of $F$ has an
acyclic subtournament on $\ge n^{\gamma}$ vertices.
\end{conj}
It would be interesting to consider high-dimensional analogs of these statements.

In Section~\ref{sub:arr} we study a hyperplane arrangement in $\mathbb R^{{n \choose d}}$ whose cells are in $1:1$ correspondence with the acyclic $n$-vertex $d$-tournaments. It is also interesting to consider the analogous arrangements in $\mathbb C^{{n \choose d}}$  and in particular ask about the fundamental group of these complex arrangements. The situation for $1$-tournaments is well understood. Namely, for $1 \le i < j \le n$ let $H_{ij}$ be the hyperplane defined by the equation $x_i=x_j$. The fundamental group $\pi_1(\mathbb C^n \setminus \underset{i<j}\cup H_{ij})$ is the so-called pure braid group, a mathematical object of great importance and interest. It is an intriguing possibility that there are interesting groups waiting to be discovered for larger $d$. For example, it would be interesting to determine the fundamental group $\pi_1(\mathbb C^{{n\choose 2}} \setminus \underset{i<j<k}\cup H_{ijk})$, where for $1\le i < j < k \le n$ we define $H_{ijk}$ as the hyperplane in $\mathbb C^{{n\choose 2}}$ whose equation is $x_{ij}+x_{jk}=x_{ik}$.

\section{Acknowledgement}

We were not sure for a while which of the many notions of acyclicity would be of greatest interest to study.
We are grateful to Roy Meshulam for helping us take the (hopefully) right decision.

\bibliographystyle{amsplain}

\begin{thebibliography}{99}

\bibitem{solymosi} N. Alon; J. Pach; J. Solymosi, {\bf Ramsey-type theorems with forbidden subgraphs}. Combinatorica 21 (2001), 155-170.

\bibitem{prob_method} N. Alon; J. Spencer,
{\bf The Probabilistic Method} (2ed). 
Wiley-Interscience Series in Discrete Mathematics and Optimization, 2000.

\bibitem{conlon} D. Conlon; J. Fox; B. Sudakov,
{\bf An improved bound for the stepping-up lemma}.
arXiv:0907.0283 [math.CO].

\bibitem{er:haj1} P. Erd\H{o}s; A. Hajnal, {\bf On Ramsey like theorems, problems and results}. Combinatorics (Proc. Conf. Combinatorial Math., Math. Inst., Oxford, 1972) , pp. 123-140, Inst. Math. Appl., Southend-on-Sea, 1972.

\bibitem{er:haj2}  P. Erd\H{o}s; A. Hajnal, {\bf Ramsey-type theorems}. Discrete Applied Mathematics 25(1989), 37-52.

\bibitem{er:moser} P. Erd\H{o}s; L. Moser, {\bf On the representation of directed graphs as unions of orderings}. Publ. Math. Inst. Hungar Acad. Sci. 9, 125-132 (1964).

\bibitem{gp} J. Goodman; R. Pollack, {\bf The complexity of point configurations}. Discrete Applied Mathematics 31 (1991) 167-180.

\bibitem{kleitman:winston} D. Kleitman;  K. Winston, {\bf On the asymptotic number of tournament score sequences}. J. Combin. Theory Ser. A 35 (1983), no. 2, 208-230.

\bibitem{leader} I. Leader; T. Tan,
{\bf Directed simplices in higher order tournaments}.
Mathematika 56 (2010), no. 1, 173-181.

\bibitem{matousek} J. Matou\v{s}ek;
{\bf Lectures on discrete geometry}.
Graduate Texts in Mathematics, 212. Springer-Verlag, New York, 2002. xvi+481 pp.

\bibitem{moon} J. Moon, {\bf Topics on tournaments}. Holt, Rinehart and Winston, New York-Montreal, Que.-London 1968 viii+104 pp.

\bibitem{stanley} R. Stanley, 
{\bf An introduction to hyperplane arrangements}.
Geometric combinatorics, 389-496, IAS/Park City Math. Ser., 13, Amer. Math. Soc. 2007.

\end{thebibliography}

\end{document}